\documentclass{amsart}
\usepackage{amsmath, amssymb,amsthm, amsfonts}
\usepackage[mathscr]{eucal}
\usepackage[all]{xy}

\def\s{\mathbb S}
\def\C{\mathbb C}
\def\bc{\mathbb C}
\def\B{\mathbb B}
\def\C{\mathbb F}

\def\V{\mathbb V}
\def\W{\mathbb W}

\def\C{\mathbb C}
\def\h{H_{\C}}
\def\R{\mathbb R}

\def\V{\mathbb V}

\def\t{\mathfrak T}
\def\P{\mathbb P}

\newtheorem{theorem}{Theorem}[section]
\newtheorem{lemma}[theorem]{Lemma}
\theoremstyle{definition}

\theoremstyle{remark}
\newtheorem{remark}[theorem]{Remark}
\numberwithin{equation}{section} \theoremstyle{plain}

\newtheorem{corollary}[theorem]{Corollary}

\numberwithin{equation}{section}

\newcommand{\secref}[1]{section~\ref{#1}}

\newcommand{\lemref}[1]{Lemma~\ref{#1}}

\newcommand{\corref}[1]{Corollary~\ref{#1}}

\begin{document}
\title[Commuting isometries of the complex hyperbolic space]{Commuting Isometries of the Complex Hyperbolic Space}
\thanks{Cao has been supported by NSF of China (No.10801107) and
NSF of Guangdong Province (No.8452902001000043)}
\author[Wensheng Cao and Krishnendu Gongopadhyay]{Wensheng Cao   and  Krishnendu
Gongopadhyay}
\address{School of Mathematics and Computational Science, Wuyi University,
Jiangmen, Guangdong 529020, P.R. China} \email{wenscao@yahoo.com.cn}
\address{Indian Institute of Science Education and Research (IISER) Mohali, Transit Campus: MGSIPAP Complex, Sector-26, Chandigarh 160019, India}
\email{krishnendu@iisermohali.ac.in, krishnendug@gmail.com}
\date{Aug 20, 2010}
\subjclass[2000]{Primary 51M10; Secondary 51F25, 20G20}
\keywords{complex hyperbolic space, isometries, centralizers}
\begin{abstract}
Let $\h^n$ denote the complex hyperbolic space of dimension $n$. The
group $U(n,1)$ acts as the group of isometries of $\h^n$. In this paper we investigate when two isometries of the complex hyperbolic space commute. Along the way we determine the centralizers.
\end{abstract}
\maketitle
\section{Introduction}\label{intr}
 Let $\h^n$ denote the complex hyperbolic space of dimension $n$. The group $U(n,1)$ acts as the group of isometries of $\h^n$.
 In this paper we ask when two isometries of the complex hyperbolic space commute. A related problem is to determine the centralizers.
 For the two dimensional complex hyperbolic space,
  Basmajian-Miner \cite[Corollaries-8.2-8.3]{bm} determined the conditions for two isometries to commute.
  In arbitrary dimension, a partial answer to this question was given by Kamiya \cite[Proposition-1.11]{kam}.
  However, the answer of Kamiya was far from being complete. In this research note,
  we classify the elements which commute with a given isometry, and along the way determine the centralizers.

Let $\V$ be a vector space over $\C$ of dimension $n+1$. Let $\V$ be 
equipped with the standard complex Hermitian form 
of signature $(n,1)$. Let $\V_-$ denote the set of all vectors with
negative length under this Hermitian form. The $n$-dimensional
complex hyperbolic space $\h^n$ is the image of $\V_-$ in the
projective space $P(\V)$. The isometry group $U(n,1)$ of the
Hermitian space $\V$ acts as the isometries of $\h^n$. The actual 
group of the isometries  is $PU(n,1)=U(n,1)/Z(U(n,1)),$ where
the center $Z(U(n,1))$ can be identified with the circle group. Thus an isometry $g$ of  $\h^n$ lifts to a unitary transformation $\tilde{g}$ in $U(n,1)$ and in the projective model of $\h^n$,
the fixed points of $g$ correspond to eigenvectors of
$\tilde{g}$.
 In the following, we shall often forget the lift, and shall use the same symbol for an isometry, as well as its lifts. An eigenvalue $\lambda$ of $g \in U(n,1)$ is called \emph{negative} if it has an eigenvector in $\V_-$.

In the ball model of the hyperbolic space, by Brouwer's fixed point theorem it follows 
that every isometry $g$ has a fixed point on the closure $\overline{\h^n}$. An isometry $g$ is called \emph{elliptic} if it has a fixed point on $\h^n$.
 It is called \emph{parabolic}, resp. \emph{hyperbolic} if it is non-elliptic, and has exactly one, resp. two fixed points on the boundary $\partial \h^n$.
  The conjugacy classification of the isometries of $\h^n$ is known essentially from the work of Chen-Greenberg \cite[Theorem 3.4.1]{chen}.
  It follows from the conjugacy classification that the elliptic and hyperbolic elements are semisimple, i.e. their minimal polynomial is a product of linear factors. The parabolic elements are not semisimple.

Let $g$ be elliptic. It follows from the conjugacy classification that all the eigenvalues have norm $1$, and it has a negative eigenvalue.
 If the negative eigenvalue has multiplicity at least $2$, then it is called a \emph{boundary elliptic}.  Otherwise, $g$ is a \emph{regular elliptic},
 i.e. all the eigenvalues are distinct. Suppose $g$ is hyperbolic. Then it has a complex eigenvalue outside the unit circle,
and one eigenvalue inside the unit circle. The other eigenvalues lie
on the circle, i.e. they have norm one. If all the eigenvalues of
$T$ are real, then $g$ is called a \emph{strictly hyperbolic}.
Otherwise, it is called a \emph{loxodromic}.  Note that a loxodromic
can be written as $g=g_r g_e$, where $g_r$ is a stretch, and $g_e$
is a boundary elliptic.
 Suppose $g$ is parabolic. If $g$ is unipotent, i.e.
all the eigenvalues are $1$,  it is called a \emph{translation}. A
translation $g$ is called \emph{vertical}, or \emph{non-vertical},
according as the minimal polynomial is $(x-1)^2$, or $(x-1)^3$.
If $g$ is a non-unipotent parabolic, then it has the Jordan
decomposition $g=g_s g_u$, where $g_s$ is semisimple, $g_u$ is
unipotent in $U(n,1)$, and $g_s$ and $g_u$ commute.  A non-unipotent
parabolic is also called an \emph{ellipto-parabolic}.  A non-vertical translation keeps invariant a unique two  dimensional totally geodesic real subspace, which is called the \emph{axis}.

A parabolic element which fixes the point `$\infty$' on the boundary, is called a \emph{Heisenberg translation}. 
Two Heisenberg translations $(\tau,t),(\sigma,s)\in \mathfrak{N}_n=
\bc^{n-1}\times \R$  are called \emph{isotropic} if ${Im}(\tau^*\sigma)=0$.
  The vertical Heisenberg translations are isotropic with any Heisenberg translation. The isotropy between two non-vertical Heisenberg translations implies that the \emph{angle}, cf. Goldman \cite[p-36]{gold},  between the axes is a non-zero real, and geometrically they are \emph{parallel}.
Suppose $S$ and $T$ are two non-vertical translations of $\h^n$. Then they can be simultaneously conjugated to two non-vertical Heisenberg translations. The axes of $S$ and $T$ are \emph{parallel} if the axes of the corresponding Heisenberg translations are parallel.  It will be clear in
Section \ref{sphs} that the parallelism of the axes, or, the isotropy between two translations,  is invariant under conjugation.
Our main theorem is the following.

\begin{theorem}\label{mainth}
\begin{enumerate}

\item{ Let $T$ acts as an elliptic isometry of $\h^n$. Then an element $S$ in $U(n,1)$ commutes with $T$ if and only if it preserves each eigenspace of $T$.   }

\item{Let $T$ acts as a hyperbolic isometry of $\h^n$. Then an element $S$ in $U(n,1)$ commutes with $T$ if and only if it preserves each eigenspace of $T$, and it acts either as the identity or a strictly hyperbolic on the complex line joining the fixed points of $T$.
 }
\item{Let $T$ be translation with fixed point $x$.   Let $S$ commute with
$T$. Then one of the following holds. 

(i) $S$ is boundary elliptic such that  $S(x)=x$ and
$T(fix(S))=fix(S)$.

(ii) $S$ is a  translation with $S(x)=x$ and, $S$ and $T$ are isotropic.

(iii) $S$ is an ellipto-parabolic with fixed point $x$, and,  it has the Jordan decomposition  $S=S_sS_u$ such that   $T(fix(S_s))=fix(S_s)$ and,   $T$ and $S_u$ are
isotropic.}

\item{Let $T=T_sT_u$ be ellipto-parabolic with fixed point $x$.   Let
$S$ commute with $T$. Then one of the following holds. 

(i) $S$ is boundary elliptic with $S(x)=x$, $T_u(fix(S))=fix(S)$ and
$T_s$ preserves each eigenspace of $S$.

(ii) $S$ is translation with $S(x)=x$,  $S(fix(T_s))=fix(T_s)$ and
$S$ and $T_u$ are isotropic.

(iii) $S$ is an ellipto-parabolic with fixed point $x$ and, it has the Jordan decomposition $S=S_sS_u$ such that  $T_u(fix(S_s))=fix(S_s), \ 
S_u(fix(T_s))=fix(T_s)$,  $T_u$ and $S_u$ are isotropic, and, $T_s$
preserves each eigenspace of $S_s$.}
\end{enumerate}
\end{theorem}
A useful corollary to the above theorem is the description of the
centralizers, up to conjugacy. Following Kulkarni \cite{kulkarni},
two elements in a group $G$ are said to be in the same
\emph{$z$-class} if their centralizers are conjugate in in $G$. The
$z$-classes provide useful stratification of the group $G$, and
provide better insight about the nature of the action of $G$ on any
set $X$, cf. Kulkarni \cite{kulkarni}. The
$z$-classes of the real hyperbolic space have been classified by
Gongopadhyay-Kulkarni \cite{g}. The  $z$-classes
of the isometries of the complex hyperbolic space is obtained from
the above theorem.
\begin{corollary}\label{zc}
Let $\mathfrak{N}_n$ denote the unipotent subgroup of a stabilizer subgroup of $U(n,1)$. Let $N$ denote the group of all isotropic translations in $\mathfrak{N}_n$.
The representatives of the $z$-classes in $U(n,1)$ are given by:

\medskip Elliptic: $Z(T)= U(r_1-1,1) \times \Pi_{j=2}^s U(r_j)$, $\Sigma_{j=1}^s r_j=n$

Hyperbolic: $\s^1 \times \R \times \Pi_{j=1}^s U(r_j)$, $\Sigma_{j=1}^s r_j =n-1$.

vertical-translation: $Z(T)= U(n-1) \ltimes  \mathfrak{N}_n$.

non-vertical translation:  $Z(T)=[\s^1 \times U(n-2)]\ltimes N$.

The $z$-class of a non-unipotent element is obtained as $Z(T)=Z(T_s) \cap Z(T_u)$, where $T_s$, $T_u$ are the semisimple and the unipotent part of $T$ in its Jordan decomposition.
\end{corollary}

In Section \ref{prel},  we review the  models of
$\h^n$, and the Heisenberg similarity transformations.  In Section
\ref{sphs}, several properties of the Heisenberg similarities are given.
Especially, we show that a unipotent isometry and an
elliptic isometry  commute if and only if they mutually fix their
fixed point sets. This is crucial to the proof of our main theorem.
In Section \ref{commuting},  we prove our main theorem, and, 
 the \corref{zc} is established in Section \ref{cen}.

\section{Preliminaries}\label{prel}
In this section we give the necessary background material on complex
hyperbolic space. More extensive facts may be found in
\cite{chen,gold,par98d}.
\subsection{The Complex Hyperbolic Space}
\subsubsection{The Ball Model}Let $\V$ be a vector space of dimension $(n+1)$ over $\C$ equipped
with the complex Hermitian form of \emph{signature} $(1,n)$,
$$\langle  z,w \rangle  =-\bar z_0w_0 + \bar z_1 w_1 + .... + \bar z_n w_n,$$
where $z$ and $w$ are the column vectors in $\C^{n,1}$ with entries
$z_1,\cdots,z_{n+1}$ and $w_1,\cdots,w_{n+1}$ respectively. Define
$$\V_0=\{z \in \V \ |\;\langle  z,z \rangle  =0\}, \ \V_+=\{z \in \V \ |\;\langle  z,z \rangle >  0\},\ \V_-=\{z \in \V \ |\;\langle z,z\rangle< 0\}.$$
 Let $\P(\V)$ be the projective space obtained from $\V$, i.e,
$\P(\V)=\V-\{0\}/\sim$ , where $u \sim v$ if there exists $\lambda$
in $\C^{\ast}$ such that $u=v \lambda$, and $\P(\V)$ is equipped
with the quotient topology.
 Let $\pi: \V-\{0\}\to \P(\V)$ denote the projection map. We define $H^n_{\C}=\pi(\V_-)$. The boundary $\partial H^n_{\C}$ in $\P(\V)$ is
$\pi(\V_0)$.

If $z=(z_0,...,z_n) \in \V_-$, the condition
$-|z_0|^2+\sum_{k=1}^n |z_i|^2 <0$ implies $z_0 \neq 0$. Therefore
we may define a set of coordinates $\zeta=(\zeta_1,...,\zeta_n)$ on
$H^n_{\C}$ by $\zeta_i(\pi(z))=z_iz_0^{-1}$. This identifies $H^n_{\C}$ with the ball
$$\B^n_{\C}=\{\zeta=(\zeta_1,...,\zeta_n)\;|\;\sum_{k=1}^n |\zeta_i|^2 <1\}. $$
With this identification the map $\pi:\V_- \to H^n_{\C}$ has the
coordinate representation $\pi(z)=\zeta$, where
$\zeta_i=z_i z_0^{-1}$. The boundary $\partial H^n_{\C}$ is
identified with
$$\s^{n-1}_{\C} =\{\zeta=(\zeta_1,...,\zeta_n)\;|\;\sum_{k=1}^n |\zeta_i^2|=1\}.$$
Let $\{f_1,...,f_n\}$ denote the standard basis of $\C^n$. Under the projection map, $f_1=(1,0,..,0)$, and
$-f_1=(-1,0,...,0)$ correspond to boundary points of $\partial H^n_{\C}$.

\subsubsection{The Siegel Domain Model}This model for $\h^n$ is obtained in the same way as above except that the Hermitian form on $\V$ under consideration is
$$
\langle z,\, w\rangle_2= w^*J z=
\overline{w_1}z_{n+1}+\overline{w_2}z_{2}+\cdots+\overline{w_n}z_{n}+\overline{w_{n+1}}z_{1},
$$
where $J$
is the Hermitian matrix $$J=\left(
                         \begin{array}{ccc}
                           0 & 0 & 1 \\
                           0 & I_{n-1} & 0 \\
                           1 & 0 & 0 \\
                         \end{array}
                       \right).$$
The two models are related by a Cayley transformation, cf.
\cite{chen,gold,par98d}. The two Hermitian forms are referred as
\emph{first} and \emph{second} Hermitian forms respectively. When
there is no confusion, we shall forget the subscript from the
Hermitian forms, and it should be clear from the context.

A {\sl unitary transformation}  $g$ is an automorphism of
$\C^{n,1}$, that is, a linear bijection such that $\langle g(
z),\,g(w)\rangle=\langle z,\,w\rangle$ for all $ z$ and $ w$ in
$\C^{n,1}$. We denote the group of all unitary transformations by
$U(n,1)$.

The isometry group of $H^n_{\C}$ is $PU(n,1)=U(n,1)/Z(U(n,1))$.
However, since the elements of $U(n,1)$ are linear, it is convenient
to deal with $U(n,1)$ rather than $PU(n,1)$.

If $g\in U(n,1)$.
Hence in the Siegel domain model,
$$
 w^*Jz=\langle z,\,w\rangle= \langle g z,\,g w\rangle = w^* g^*J g z
$$
for all $ z$ and $w$ in $\C^{n,1}$. Letting $ z$ and $w$ vary over a
basis for $\C^{n,1}$, we see that $J= g^*J g$. From this we find $
g^{-1}=J^{-1} g^*J$. That is:
\begin{equation}\label{hform}
g^{-1}=\left(
  \begin{array}{ccc}
     \bar{d}& \beta^*& \bar{b} \\
    \delta & A^*& \gamma\\
    \bar{c}& \alpha^*& \bar{a}\\
    \end{array}
\right) \ \ \mbox{for}\ \
 g=\left(
  \begin{array}{ccc}
     a& \gamma^*& b \\
    \alpha & A& \beta\\
    c & \delta^* & d\\
    \end{array}
\right)\in U(n,1),\end{equation} where $A$ is an $(n-1)\times (n-1)$
matrix, $\alpha,\beta, \gamma, \delta$ are column vectors in
$\C^{n-1},$  and $a,b,c,d\in \C$.

\subsection{The Heisenberg Similarity Transformations}

The $(2n-1)$-dimensional Heisenberg group $\mathfrak{N}_n$ is
$\bc^{n-1}\times \R$ with the group law
$$(\zeta_1,v_1)(\zeta_2,v_2)=(\zeta_1+\zeta_2,v_1+v_2+2\textrm{Im}(\zeta_2^*\zeta_1)).$$
The boundary
of $\h^n$ may be identified with the one-point
compactification of the Heisenberg group.  We define {\it
horospherical coordinates} on the complex hyperbolic space as follows.
$$\psi: \mathfrak{N}_n\times \R_{+}\to \V_0\cup \V_{-}$$
$$\psi(\zeta,v,u)=\left(
                     \begin{array}{c}
                       (-|\zeta|^2-u+iv)/2 \\
                       \zeta \\
                       1 \\
                     \end{array}
                   \right),\ \  \psi(\infty)=\left(
                                           \begin{array}{c}
                                             1 \\
                                             0 \\
                                             \vdots \\
                                             0 \\
                                           \end{array}
                                         \right).
$$
We also define the origin $o$ to be the point in $\partial H^n_{\C}$
with \emph{horospherical} coordinates $(0, 0, 0),$  that is
$$\psi(o)=(0,\cdots,0,1)^t.$$ In this way the $\partial H^n_{\C}$ is
identical with $\mathfrak{N}_n\times \{0\}\cup \infty$.

The unitary group $U(n-1)$ acts on   horospherical coordinates as
\emph{Heisenberg rotation}, whose action is given by
$$R_{U}: (\zeta,v,u) \to (U\zeta,v,u).$$  The corresponding matrix
in $U(n, 1)$ acting on $\C^{n,1}$ is
$$R_{U}=diag(1,U,1).$$
The fixed point set of such a map is a totally geodesic subspace of
 $H^n_{\C}$ passing through  $o$
and  $\infty$.

The positive real numbers $r \in \R^{+}$  act as  \emph{Heisenberg
dilation} $D_r$ on  horospherical coordinates, whose  action is
given by
$$D_r:(\zeta,v,u)\to
(r\zeta,r^2v,r^2u).$$  As a matrix $D_r$ is given by
$$D_r=diag(r,I_{n-1},\frac{1}{r}).$$

  The $(\tau, t) \in
\mathfrak{N}_n$ acts  as \emph{Heisenberg translation} on
horospherical coordinates, whose action is given by
$$T_{(\tau,t)}:(\zeta,v,u)\to ((\tau,t)(\zeta,v),u)=(\tau+\zeta,v+k+2\textrm{Im}(\zeta^*\tau),u).$$  As a matrix $T_{(\tau,t)}$ is given by
\begin{equation} \label{translation} \left(
     \begin{array}{ccc}
       1 & -\tau^* & \frac{-|\tau|^2+ti}{2} \\
       0 & I_{n-1} & \tau \\
       0 & 0 & 1 \\
     \end{array}
   \right).
\end{equation} When $\tau=0$, $T_{(\tau,t)}$ is a vertical
Heisenberg translation. Otherwise, it is a {non-vertical
Heisenberg translation}. It is easy to see that the subgroup $\t$ form
by vertical Heisenberg translations  is the center of
$\mathfrak{N}_n$ and $\t \approx \R$.

 A parabolic element  is called \emph{screw
parabolic} if it is  the product of a Heisenberg translation and a
Heisenberg rotation.  When $n\geq 3$, we will show that a screw
parabolic  which is the product of a non-vertical Heisenberg
translation and a Heisenberg rotation can not be conjugated to a
screw parabolic which is the product of a vertical  Heisenberg
translation and Heisenberg rotation in  Section \ref{sphs}.


Let $U(n,1)_{\infty}$ denote the isotropy group of  $U(n,1)$ at $\infty$.
Then
\begin{equation}\label{isot}U(n,1)_{\infty}=[\s^1 \times \R^+ \times  U(n-1)] \ltimes  \mathfrak{N}_n,\end{equation}
where $\ltimes$ denotes the semidirect product of groups. Thus every
element $g$ in $U(n,1)_{\infty}$  can be uniquely  written as $g=D_r
R_UT_{(\tau,t)}$ , where $R_U$ is the  elliptic part, $D_r$ the
hyperbolic part and $T_{(\tau,t)}$ the \emph{unipotent} part of this
Heisenberg  isometry. Elements in $U(n,1)_{\infty}$ are called the
\emph{Heisenberg similarity transformations}. If $D_r$ is the identity in
the above decomposition, $g$ is called a \emph{Heisenberg isometry}. It is
easy to see that elements $g$ fixing $o$ and $\infty$ are of the
form
\begin{equation}\label{eq4}
g=diag(\mu, A, \lambda)
\end{equation}
where $\bar \mu \lambda=1$ and $A \in U(n-1)$.

\section{Commuting elements of Heisenberg translations}\label{sphs}
Throughout this section, we use the second Hermitian form of the hyperbolic space.
\begin{lemma}\label{fpl} Let $T,S\in PU(n,1)$.  If  $S$ commutes with $T$  then  $$T(fix(S))=fix(S),\, \,
S(fix(T))=fix(T).$$
\end{lemma}

\begin{proof} Let $x\in fix(S)$.  Then $ST(x)=TS(x)=T(x)$, which
implies that $T(fix(S))\subset fix(S)$. Also we have
$T^{-1}(x)=T^{-1}S(x)=ST^{-1}(x),$ which implies that
$T^{-1}(fix(S))\subset fix(S)$. Thus $T(fix(S))=fix(S)$.
\end{proof}

By matrix computation, we have the following observation
\begin{equation} T_{(\tau, t)}R_U=R_UT_{(U^*\tau, t)}.\end{equation}
Using this we have the following.
\begin{lemma}\label{hcomu} Let $T=R_UT_{(\tau,t)}$ and $S=R_VT_{(\sigma,s)}$.
Then $S$ commutes with $T$ if and only if
$$VU=UV,\, \,  (V^*\tau,t)(\sigma,s)=(U^*\sigma,s)(\tau,t).$$
\end{lemma}

\begin{lemma}
 Let $T=R_UT_{(\tau,t)}$ and $T_{(\tau,t)}\neq I $.  If $U\tau=\tau$ then $T$ is
 parabolic.
 \end{lemma}

\begin{proof} Suppose that $U\tau=\tau$ and $T$ is elliptic.  Then
$T$ is boundary elliptic, and by its action on horospherical
coordinate  there exists point $(\zeta,v,u)$ in $H_{\C}^n$ such that
$$U(\tau+\zeta)=\zeta, t+2\textrm{Im}(\zeta^*\tau)=0.$$
Hence we have that
$$\zeta^*\tau=\zeta^*U^*U\tau=(\zeta^*-\tau^*)\tau=\zeta^*\tau-|\tau|^2.$$
This implies that $\tau=0$ and $t=0$.  This contradiction implies
$T$ is parabolic.
\end{proof}

\begin{remark}
Note that $T=R_UT_{(\tau,t)}$ with $U\tau=\tau$ and $R_U\neq
I_{n-1}$ can also be conjugated to the form $R_UT_{(\tau_1,t_1)}$
such that $U\tau_1\neq \tau_1$.  We give the reason as follows. Let
$\V_1(U)=\{x\in \C^{n-1}\, | Ux=x\}.$ For $\tau\neq 0$,
$U\tau=\tau$,  by linear algebraic theory,  the matrix equation
$$(U-I)x=\tau$$
has no solution $x$.  This is $\sigma+\tau-U\sigma\neq 0$. When
$n\geq 3$, there exists  nonzero $\sigma\notin \V_1(U)$. For such
$\sigma$,  $T_{(\sigma,s)}TT_{(\sigma,s)}^{-1}$ is
\begin{equation}\label{conju1}
R_UT_{(U^*\sigma+\tau-\sigma,t-2\textrm{Im}(\sigma^*U\sigma)+4\textrm{Im}(4\tau^*\sigma))}=
T_{(\sigma+\tau-U\sigma,t-2\textrm{Im}(\sigma^*U^*\sigma)+4\textrm{Im}(4\tau^*\sigma))}R_U\end{equation}
where $\tau_1=U^*\sigma+\tau-\sigma\neq
(\sigma+\tau-U\sigma)=U\tau_1$. We also have
\begin{equation}\label{conju2}R_VTR_V^{-1}=R_{VUV^*}T_{(V\tau,t)}=T_{(V\tau,t)}R_{VUV^*},\end{equation}
\begin{equation}\label{conju3}D_rTD_r^{-1}=R_UT_{(r\tau,r^2t)}=T_{(r\tau,r^2t)}R_U.\end{equation}
The above two conjugations can not make the unipotent part of
$T=R_UT_{(\tau,t)}$ to be vertical Heisenberg translation.  The
above observations imply that, when $n\geq 3$, we can not conjugate
a screw parabolic $T=R_UT_{(\tau,t)}$ with $U\tau=\tau\neq 0$   to
the form of $R_{U'}T_{(0,t')}$.  By the three conjugations
(\ref{conju1}), (\ref{conju2}),(\ref{conju3}), we  know that the
definition of isotropic in Section \ref{intr} is well defined.  We
also mention that a boundary elliptic element $R_vT_{(\sigma,s)}$
can be conjugated in $U(n,1)_{\infty}$ to diagonal form.

By the above, it also follows that, for $n\geq 3$,  there are two types of
screw parabolic elements $T=R_UT_{(\tau,t)}$ with $U\tau=\tau$,
$\tau \neq 0$ and $\tau=0$, respectively.
\end{remark}

In what follows, we concern about which element can commute with a
unipotent parabolic element $T$.  By conjugation, we may assume that
$T$ fixes $\infty$,  that is $T=R_UT_{(\tau,t)}$.   By Lemma
\ref{fpl}, if $S$ commutes with $T$, then we must have
$$T(fix(S))=fix(S),\, \, S(\infty)=\infty.$$  If $S$ is hyperbolic, then by
$T(fix(S))=fix(S)$, we get a contradiction. This implies that
$S$ is Heisenberg isometry which can be written uniquely
  as $S=R_VT_{(\sigma,s)}$.

\begin{lemma}\label{uulem}
 Let $T$ and $S$  be unipotent.  Then $ST=TS$ if and only
 if $S$ and $T$ are isotropic.
\end{lemma}

\begin{proof} Suppose  $TS=ST$.  By Lemma \ref{fpl}, they have a
common fixed point $x$.  Let $K\in U(n,1)$ such that $K(x)=\infty$.
Then $KTK^{-1}$ and $KSK^{-1}$ must be of the forms $T_{(\tau,t)}$
and $T_{(\sigma,s)}$.  It follow from $TS=ST$ and
(\ref{translation}) that  $S$ and $T$ are isotropic. The necessity
is obvious from the definition of isotropic.
\end{proof}

\begin{lemma}\label{uelem}
 Let $T$ be  unipotent and  $S$  elliptic.  Then $ST=TS$ if and only
 if $T(fix(S))=fix(S)$ and $S(fix(T))=fix(T)$.
\end{lemma}

\begin{proof}
By Lemma \ref{fpl}, we only need to prove the
necessity.  Without loss of generality, let $T=T_{(\tau,t)}$ and
$S=R_VT_{(\sigma,s)}$ with $V\sigma\neq \sigma$. The action of $S$
on horospherical coordinate $(\zeta,v,u)$ is
$$S:(\zeta,v,u)\to (V\sigma+V\zeta,v+s+2\textrm{Im}(\zeta^*\sigma),u).$$
Let $(\zeta,v,u)\in H_{\C}^n$ be a fixed point of $S$, then
$$V(\sigma+\zeta)=\zeta, s+2\textrm{Im}(\zeta^*\sigma)=0.$$
 If $T(fix(S))=fix(S)$ then $T((\zeta,v,u))=(\tau+\zeta,t+v+2\textrm{Im}(\zeta^*\tau),u)\in
 fix(S)$, that is
 $$V(\sigma+\tau+\zeta)=\tau+\zeta, s+2\textrm{Im}((\tau+\zeta)^*\sigma)=0.$$
This implies that $V\tau=\tau, \textrm{Im}(\tau^*\sigma)=0$.  By
Lemma \ref{hcomu}, this implies that $ST=TS$.
\end{proof}

\section{ Proof of the Main Theorem}\label{commuting}
In this section, we shall prove our main theorem case by case. Mainly
using Lemma \ref{hcomu} and the proof of our main theorem, we shall
determine the centralizers up to conjugacy.

\subsection{The semisimple cases} To classify the commuting elements of the semisimple isometries, we use the first Hermitian form $\langle, \rangle$.
It is customary to call a
subspace $\W$ (i) \emph{time-like} if the restriction ${\langle ,
\rangle}|_{\W}$ is nondegenerate and indefinite,  (ii)
\emph{space-like} if ${\langle , \rangle}|_{\W}>0$, (iii)
\emph{light-like} if ${\langle , \rangle}|_{\W}=0$. In particular, a
vector  $v\in \V$ is \emph{time-like}, \emph{space-like} or
\emph{light-like} according as $\langle  v,v \rangle  <0 $, $\langle
v, v \rangle >0$ or $\langle v, v \rangle=0$.

Suppose $T$ is semisimple, i.e. elliptic or hyperbolic. Suppose $S$ commutes with $T$. Let $\lambda$ be an eigenvalue of $T$ and let $\V_{\lambda}$
be the eigenspace to $\lambda$. Let $v \in \V_{\lambda}$. Then $TS(v)=ST(v)=S(\lambda v)=\lambda S(v)$, i.e.  $S(v) \in \V_{\lambda}$.
Thus  each eigenspace of $T$ is $S$-invariant, and each eigenspace of $S$ is $T$-invariant.

 Let $\lambda$ and $\mu$ be two distinct  eigenvalues of $T$.   For $v \in \V_{\lambda}$, $w \in \V_{\mu}$,
  note that $$\langle v, w \rangle=\langle Tv, Tw \rangle=\bar \lambda \mu \langle v, w \rangle, $$
thus if $\bar \lambda \mu \neq 1$, then $\V_{\lambda}$ and $\V_{\mu}$ are orthogonal to each-other.

\medskip \noindent {(1)} Let $T$ be elliptic. Suppose each eigenspace of $T$ is $S$-invariant.
We claim that $S$ commutes with $T$. For, suppose the distinct eigenvalues of $T$ are $e^{i \theta_j}$, $j=1,...,s$,
where $e^{i \theta_1}$ is assumed to be time-like and the rest are space-like. For each $j$, let $\V_j$
be the eigenspace to $e^{i \theta_j}$.
Then $\V$ has the orthogonal decomposition: $\V=\oplus_{j=1}^s \V_j$. Let $Q_j$ denote the Hermitian form ${\langle, \rangle}|_{\V_j}$.
Since $S$ is an isometry of $\V$,  each $\V_j$ is $S$-invariant, hence $S$ preserves the orthogonal decomposition and also the restricted metric on each of the summands.
 Thus $S$ commutes with $T$ if and only if $S|_{\V_j}$ commutes with $T|_{\V_j}$. Let $Q_j$ denote the
 restriction of $Q$ to $\V_j$. Since $T|_{\V_j}=e^{i \theta_j} I_{\V_j}$, it is a central element in $U(\V_j, Q_j)$.
Hence any $S|_{\V_j}$ in $U(\V_j, Q_j)$ commute with $T|_{\V_j}$. This establishes the claim.

\medskip \noindent {(2).} Suppose $T$ is hyperbolic. Let $L_r$, resp, $L_{r^{-1}}$ denote the one-dimensional eigenspaces
corresponding to the light-like eigenvalues $r$, resp. $r^{-1}$. Let $\V_r=L_r + L_{r^{-1}}$.
Let $\V_j$ denote the eigenspace of $T$ corresponding to the space-like eigenvalue $e^{i \theta_j}$.
In this case $\V$ has the orthogonal decomposition:
$$\V=\V_r \bigoplus \oplus_{j=1}^s  \V_j,$$
where $\V_r=L_r + L_{r^{-1}}$. Clearly $S$ keeps $\V_r$ invariant. For $j=r,1,...,s$, Let $Q_j$ denote the Hermitian form ${\langle, \rangle}|_{\V_j}$.
 As above, we see that $S|_{\V_j}$ commutes with $T|_{\V_j}$ for each $j=1,...,s$.  Hence $S$ commutes with $T$ if and only if $S|_{\V_r}$ commutes with $T|_{\V_r}$.
 Since $S$ keeps the eigenspaces of $T$ invariant,  it keeps the light-like lines $L_r$ and $L_{r^{-1}}$ invariant.
Hence $S|_{\V_r}$ commutes with $T|_{\V_r}$ if and only if it is either a strictly hyperbolic or a central element in $U(\V_r, Q_r)$.  Consequently the assertion follows.

\subsection{Parabolic case} We shall use the Siegel domain model, i.e. the second Hermitian form to deal with the parabolic case. Let $T$ be a parabolic isometry.
 Let $S$ commute with $T$. Suppose $S$ is semisimple. It can not be hyperbolic: for otherwise, $T$ would have two fixed points on the boundary.
 So $S$ must be elliptic. We assert that $S$ can not be a regular elliptic. For otherwise, $T$  fixes  the unique fixed point of $S$, which is impossible.
 Hence $S$ must be a boundary elliptic.

By conjugation, we assume its unique fixed point to be $\infty$, i.e. $T$ is a Heisenberg isometry.
 If $T$ is ellipto-parabolic, it has the unique
Jordan decomposition $T=T_sT_u=T_uT_s$, where $T_s$ and $T_u$ are
its semi-simple and unipotent components, respectively. From the Jordan decomposition,
it follows that an element commutes with $T$ if and only if it commutes with both $T_s$ and $T_u$. Now the assertions follow from \lemref{fpl}, \lemref{uulem} and
\lemref{uulem}.

This completes the proof of the main theorem.

\subsection{The Centralizers: Proof of \corref{zc}} \label{cen}
\subsubsection{Elliptic elements} Let $T$ be elliptic. Then we have seen that $\V$ has an orthogonal decomposition:
$\V=\oplus_{j=1}^s\V_j$, where for each $j$, $\V_j$ is the
eigenspace to $e^{i \theta_j}$, and $\theta_1$ be the negative
eigenvalue. Let $\dim \V_j=r_j$. For each $j$, $T|_{\V_j}$ is a
central element in $U(\V_j, Q_j)$,which we identify with
$U(r_1-1,1)$, resp. $U(r_j)$ for $j=1$  resp. $>1$. Hence
$$Z(T)= \Pi_{j=1}^s Z(T|_{\V_j})=U(r_1-1,1) \times \Pi_{j=2}^s U(r_j).$$
\subsubsection{Hyperbolic elements}  Let $T$ be hyperbolic. Then $\V$ has the orthogonal decomposition
$\V=\V_r \bigoplus \oplus_{j=1}^s  \V_j$, where $V_r$ is the
$2$-dimensional subspace spanned by the light-like eigenvectors, and
$\V_j$ are the eigenspaces spanned by space-like eigenvectors to
eigenvalues $e^{i \theta_j}$, $j=1,..,s$. Hence $Z(T)=Z(T|_{\V_r})
\times \Pi Z(T|_{\V_j})$. Let $\dim \V_j=r_j$. As above, we identify
it with $U(r_j)$. Note that $T|_{\V_r}$ acts as a hyperbolic element
in $U(\V_r, Q_r) \equiv U(1,1)$, hence $Z(T|_{\V_r})=\s^1 \times
\R$. Hence
$$Z(T)=\s^1 \times \R \times \Pi_{j=1}^s U(r_j). $$
\subsubsection{Parabolic elements}
First note that the centralizer of $U(n,1)$ is
$$\s^1=\{\mu I\ | \ |\mu|=1\}.$$

\medskip (1)  Suppose $T$ is a vertical translation.   Since $T$ can be conjugated to vertical Heisenberg translation, by Lemma \ref{hcomu}, up to
conjugation in $PU(n,1)$,
$$Z(T)= U(n-1) \ltimes  \mathfrak{N}_n. $$

Alternatively, using the linear group $U(n,1)_{\infty}$, we choose the following subgroups:
\begin{equation}\label{adeq1}\hat \s^1=\bigg\{ g \in U(n,1) \ | \ g=diag(\mu, I_{n-1}, \mu), \ |\mu|=1 \bigg\}.\end{equation}
$$\hat U(n-1)=\bigg\{g \in \hat U(1,n) \ | \ g=diag(1, A, 1), \ A \in U(n-1)\bigg\}.$$
By direct computation, following the approach of Chen-Greenberg \cite[section-4.2]{chen},
one can see that every element $S$ in $Z(T)$ can be written uniquely as $S=ABP$, $A \in \hat \s^1$, $B \in \hat U(n-1)$, $P \in \mathfrak{N}_n$.

Thus we have, up to conjugation in $U(n,1)$,
$$Z(T)=[\s^1 \times U(n-1)]\ltimes \mathfrak{N}_n.$$

(2) Suppose $T$ is a non-vertical translation. Let $T=T_{(\tau,t)}$ and   $S \in Z(T)$.
Let
$$N=\{T_{\sigma,s} \in \mathfrak{N}_n | \ \tau^{\ast}\sigma \in \R \}.$$
Note that $\t$ is a subgroup of $N$. It follows from \secref{sphs} that,  up to conjugation,
 $S=R_V T_{(\sigma,s)}$ with $\tau\neq 0$, $V \tau=\tau$, and $\tau^{\ast} \sigma \in \R$.  Using linear algebraic methods,
 there exists a unitary matrix $U$ such that
$U\tau=(|\tau|,0,\cdots,0)^t$, that is $UVU^*e_1=e_1.$  Thus
$$ E_{\tau}=\{V\in U(n-1)\, |\,   V\tau=\tau\}\approx U(n-2),$$
which implies that, up to conjugation in $PU(n,1)$,
$$Z(T)= U(n-2) \ltimes  N. $$

Alternatively,  in the linear group $U(n,1)_{\infty}$, let
$$\hat E_a=\{g \in U(1, n)_{\infty} \ | \ g=diag(\mu, A, \mu), \ A a=\mu a\}. $$
By direct computation, one can see that every element $S$ in $Z(T)$ can be written uniquely as $S=BP$, where $B \in \hat E_{a}$, $P \in N$.
Further we see that $\hat E_a$ can be identified with the group $\s^1 \times U(n-1)$.
Hence, up to conjugation in $U(n,1)$,
$$Z(T)=[\s^1 \times U(n-2)]\ltimes N.$$
\subsubsection{A topological description of the centralizer of $T_{\sigma, s}$}
Let $\langle, \rangle_o$ denote the Euclidean norm in $\C^{n-1}$,
i.e. $\langle x, y \rangle_o=x^* y$. Thus   we can identify
$$N/\t=\{c \in \C^{n-1} \ | \  \langle a,c  \rangle_o  \in \R\}.$$
We identify $\C^{n-1} = \R^{2(n-1)}$, and let $\langle, \rangle_o$ be the usual Euclidean norm on $\R^{2(n-1)}$.
Then $\langle a, c \rangle_o=r$ is a $(2n-3)$-dimensional affine hyperspace orthogonal to $a$, and hence $N/\t \approx \R \times \R^{2(n-2)}\approx \R \times \C^{n-2}$.
Thus
$$N\approx[\R \times \C^{n-2}]\ltimes \t \approx [\R \times \C^{n-2}]\ltimes \R.$$
Hence
$$Z(T)\approx [\s^1 \times U(n-2)]\ltimes[(\R \times \C^{n-2}) \ltimes \R].$$


\begin{thebibliography}{99}
\bibitem{bm} A. Basmajian  and R. Miner, {\it Discrete  subgroup of Complex hyperbolic motions}.
Invent. Math., 131(1)(1998), 85-136.

\bibitem{chen} S. S. Chen  and L. Greenberg,  {\it Hyperbolic spaces}. Contributions
to analysis, New York: Academic Press, 1974, 49-87.

\bibitem{gold} W. M. Goldman, {\it Complex hyperbolic geometry}. Oxford University Press,  1999.

\bibitem{g} K. Gongopadhyay and R. S. Kulkarni, {\it $z$-Classes of isometries of the hyperbolic space}.Conform. Geom. Dyn. 13  (2009), 91--109.

\bibitem{kam}S. Kamiya,  {\it Notes on elements of $U(1,n;C)$}.  Hiroshima Math. J.,
21(1991), 23-45.

\bibitem{kulkarni} R. S. Kulkarni, {\it Dynamical types and conjugacy classes of centralizers in groups}.  J. Ramanujan Math. Soc. (1) {\bf 22}  (2007),  35--56.


\bibitem{par98d} J. R. Parker,  {\it On the volumes of cusped, complex hyperbolic manifolds
and orbifolds}.  Duke Math. J.,   94 (3) (1998),  433-464.


\end{thebibliography}
\end{document}